\documentclass[11pt]{amsart}
\usepackage{amsmath}
\usepackage{amsthm}
\usepackage{amsbsy}
\usepackage{amssymb}
\usepackage{amsfonts}
\usepackage[dvips]{lscape}
\usepackage{xcolor}
\usepackage{amscd}
\usepackage[all,cmtip]{xy}
\usepackage{euscript}
\usepackage{parskip}
\usepackage{enumerate}
\usepackage[all,cmtip]{xy}
\usepackage{tikz-cd}

\usepackage[margin=1in]{geometry}
\usepackage[expansion=false]{microtype}

\usepackage[pdfusetitle,unicode,hidelinks]{hyperref}

\theoremstyle{plain}

\theoremstyle{plain}
\newtheorem{theorem}{Theorem}[section]

\newtheorem{lemma}[theorem]{Lemma}
\newtheorem{corollary}[theorem]{Corollary}

\theoremstyle{remark}
\newtheorem{remark}[equation]{Remark}

\newtheorem{question}[theorem]{Question}

\theoremstyle{definition}

\newcommand{\cal}{\EuScript}

\newcommand{\id}{\textup{id}}

\let\lim=\relax

\DeclareMathOperator*{\lim}{lim}

\newcommand{\tr}{\textup{tr}}
\newcommand{\Fr}{\textup{Fr}}
\newcommand{\Rm}{\textup{Rm}}

\newcommand{\End}{\textup{End}}
\renewcommand{\contentsname}{}
\begin{document}
\title{Riemannian structures and point-counting}
\author{Masoud Zargar}
\renewcommand{\contentsname}{}
\begin{abstract}
Suppose $(X_n)$ is a sequence of positive-dimensional smooth projective complete intersections over $\mathbb{F}_q$ with dimensions bounded from above and with characteristic zero lifts $(\tilde{X}_n)$ to smooth projective geometrically connected varieties. Suppose each complex variety $\tilde{X}^{an}_n$ has (underlying real manifold equipped with) a Riemannian metric $g_n$ of sectional curvature at least $-\kappa_n^2$, $\kappa_n\geq 0$, and diameter at most $D_n$. In this note, we show that if 
\[\lim_{n\rightarrow\infty}\min\{d:X_n(\mathbb{F}_{q^d})\neq\emptyset\}=+\infty,\]
then $\kappa_nD_n\rightarrow+\infty$. We deduce this theorem by proving a more general theorem estimating the number of points over finite fields of the above varieties in terms of the sectional curvature and diameter of Riemannian structures on the analytification of characteristic zero lifts. We also prove a version of this estimate for non-projective varieties equipped with complete metrics of non-negative sectional curvature. Finally, we prove a characteristic zero analogue of the above estimate by relating fixed-points of algebraic endomorphisms of smooth projective complex complete intersections to Riemannian structures on them.
\end{abstract}
\maketitle
\vspace{-1cm}
\section{Introduction}\label{intro}
The relation between topology, arithmetic, and dynamics is not new. For example, the genus of a smooth projective curve over a number field has a great influence on its arithmetic. Indeed, Mordell's conjecture, now a theorem thanks to Faltings~\cite{Faltings}, states that a smooth projective curve of genus at least $2$ over a number field $k$ has finitely many  $k$-points. Furthermore, beginning with the crystallization of the ideas of Weil by the Grothendieck school of algebraic geometry, we know that topological data is useful in deducing arithmetic information about varieties over finite fields. For example, the Grothendieck-Lefschetz fixed-point theorem establishes an intimate connection between the action of Frobenius on $\ell$-adic cohomology groups and the number of points over finite fields of varieties. On the other hand, one of the main themes in Riemannian geometry is the deduction of (global) topological information from (mostly local) geometric information. For example, certain restrictions on various forms of curvature put restrictions on the fundamental group, Betti numbers, and sometimes even homeomorphism type. The first instance of this phenomenon is the Gauss-Bonnet theorem. Therefore, it is not surprising that Riemannian structures on varieties have something to say about their arithmetic. In this paper, instead of topology, we relate Riemannian structures to the dynamics and arithmetic of our varieties. More precisely, we show that simultaneous bounds on the diameter and sectional curvature of the underlying real manifold of the complexification of smooth characteristic zero lifts of varieties have arithmetic (over finite fields) as well as dynamical consequences.\\
\\
One of the main tools for proving the existence of fixed-points of endomorphisms in topology is the Lefschetz fixed-point theorem. Its analogue in arithmetic, the Grothendieck-Lefschetz fixed-point theorem or more generally the Grothendieck trace formula, is used to count the number of points of a variety over finite fields, that is, the number of fixed-points of the powers of the Frobenius endomorphism. One of the techniques used in the theory of exponential sums is combining the Grothendieck trace formula, the Weil conjectures, and bounds on the sum of $\ell$-adic Betti numbers to estimate the number of $\mathbb{F}_q$-points on relevant varieties. This technique has led to many nontrivial bounds. When estimating the constants that show up in the estimates, it is useful to bound the sum of the $\ell$-adic Betti numbers from above. A few such bounds on sums of $\ell$-adic Betti numbers have been proved by N.Katz~\cite{Katzbetti}. On the transcendental side, we also have analogues of such tools that, in fact, precede the above arithmetic tools. Before Deligne's proof of the Weil conjectures, Serre proved a characteristic zero analogue for algebraic endomorphisms of smooth projective complex varieties~\cite{Serre}. Moreover, Gromov has proved an upper bound on the sum of Betti numbers of closed connected Riemannian manifolds in terms of the geometry of the manifold. This upper bound depends only on the dimension, sectional curvature, and diameter of the closed connected Riemannian manifold. By combining these tools, we estimate the number of points over finite fields in terms of properties of the Riemannian structure. Furthermore, we also prove a result for varieties that are not necessarily projective (and so do not give rise to \textit{compact} Riemannian manifolds). This is done by considering smooth varieties that, in addition to some other conditions, have complex analytifications that admit \textit{complete} Riemannian metrics with nonnegative sectional curvature. Similar theorems can be proved for Ricci curvature; see Remark~\ref{Ricci} below. As an application, we prove the following qualitative result relating points over finite fields to Riemannian structures on the analytification of characteristic zero lifts. Throughout this paper, Riemannian structures on a smooth complex variety are Riemannian structures on the underlying real manifold. Furthermore, we say a variety over a finite field $\mathbb{F}_q$ has a smooth characteristic zero lift if it is obtained, say, as the fiber (over $\mathbb{F}_q$) of a smooth family over the ring of integers of a number field with possibly finitely many primes of the ring of integers inverted.
\begin{theorem}\label{main}Suppose $(X_n)$ is a sequence of positive-dimensional smooth projective complete intersections over $\mathbb{F}_q$ with dimensions bounded from above and with characteristic zero lifts $(\tilde{X}_n)$ to smooth projective geometrically connected varieties. Suppose each complex variety $\tilde{X}^{an}_n$ is equipped with a Riemannian metric $g_n$ of sectional curvature at least $-\kappa_n^2$, $\kappa_n\geq 0$, and diameter at most $D_n$. If 
\[\lim_{n\rightarrow\infty}\min\{d:X_n(\mathbb{F}_{q^d})\neq\emptyset\}=+\infty,\]
then $\kappa_nD_n\rightarrow+\infty$. 
\end{theorem}
We may normalize the Riemannian structures so that the diameters are bounded from above by $1$ or that the sectional curvatures are bounded from below by $-1$. In this case, then, morally, this theorem says that if for each finite field extension of $\mathbb{F}_q$ there is always a variety among the infinite family that avoids points over the finite field extension, the analytifications of the characteristic zero lifts must become arbitrarily negatively curved or become arbitrarily large in diameter, respectively, no matter what Riemannian metrics we choose as above.\\
\\
This theorem is a corollary of the following more general theorem giving an estimate on the number of $\mathbb{F}_q$-points of our varieties in terms of the geometry of an underlying Riemannian structure.
\begin{theorem}\label{A}
For each $n\geq 1$, there is a constant $C=C(n)$ such that the following is true. Suppose $X$ is an $n$-dimensional smooth projective complete intersection over $\mathbb{F}_q$ with a characteristic zero lift $\tilde{X}$ whose generic fiber is a geometrically connected smooth projective variety. Furthermore, suppose $\tilde{X}^{an}$ admits a Riemannian metric whose sectional curvature is at least $-\kappa^2$, $\kappa\geq 0$, and whose diameter is at most $D$. Then 
\[||X(\mathbb{F}_q)|-|\mathbb{P}^n(\mathbb{F}_q)||\leq C^{1+\kappa D}q^{n/2}.\]
In particular, there is a constant $K=K(n)$ depending only on $n$ such that $X$ has an $\mathbb{F}_q$-point if $q>K^{1+\kappa D}$.
\end{theorem}
\begin{remark}
It is easy to see that $|\mathbb{P}^n(\mathbb{F}_q)|=1+q+\hdots+q^n$.
\end{remark}
\begin{remark}Note that without restrictions on the geometry, Theorem~\ref{A} is not true even for curves. Indeed, there is no constant $C$ such that every smooth projective curve over $\mathbb{Q}$ satisfies
\[||X(\mathbb{F}_q)|-|\mathbb{P}^1(\mathbb{F}_q)||\leq Cq^{1/2},\]
where $p$ is any prime of good reduction. Indeed, let $q$ be a power of a prime $p$, and consider the smooth projective Fermat curve $X_q:=\textup{Proj}\left(\frac{\mathbb{Q}[X,Y,Z]}{(X^{q+1}+Y^{q+1}-Z^{q+1})}\right)$. These have good reduction at $p$. It is known that over $\mathbb{F}_{q^2}$, $X_q$ ("basechanged" to $\mathbb{F}_q$) satisfies the Weil bound. Indeed, over $\mathbb{F}_{q^2}$, these curves have genus $g=q(q-1)/2$, and have $1+q^3$ points. Consequently,
\[\frac{||X_q(\mathbb{F}_{q^2})|-|\mathbb{P}^1(\mathbb{F}_{q^2})||}{q}=\frac{|(1+2g(-q)+q^2)-(1+q^2)|}{q}=2g\rightarrow\infty\]
as $g\rightarrow\infty$ (equivalently, as $q\rightarrow\infty$).
\end{remark}
As a corollary, we obtain the following.
\begin{corollary}\label{corA}
For each $n\geq 1$, there is a constant $C=C(n)$ such that the following is true. Suppose $X$ is an $n$-dimensional smooth projective complete intersection over $\mathbb{F}_q$ with a characteristic zero lift $\tilde{X}$ whose generic fiber is a geometrically connected smooth projective variety. Furthermore, suppose $\tilde{X}^{an}$ admits a Riemannian metric with non-negative sectional curvature. Then there is a constant $K=K(n)$ depending only on $n$ such that $X$ has an $\mathbb{F}_q$-point if $q>K$.
\end{corollary}
The second arithmetic theorem is for smooth $\mathbb{Q}$-varieties that are not necessarily projective. It is not a generalization of Theorem~\ref{A} because it assumes the metric is complete and of nonnegative sectional curvature.
\begin{theorem}\label{B}
For each dimension $n\geq 1$, there is a constant $C=C(n)$ such that the following is true. Suppose we have an $n$-dimensional projective $\mathbb{Q}$-variety $\overline{X}$ with a hyperplane section $D$ such that $X:=\overline{X}\setminus D$ is geometrically connected and smooth and such that $X^{an}$ admits a complete Riemannian metric of nonnegative sectional curvature. Furthermore, suppose $p$ is a prime of good reduction for $X$ and that $\overline{X}_{\overline{\mathbb{F}}_q}$ is a complete intersection, not necessarily smooth. Then
\[||X(\mathbb{F}_q)|-q^n|\leq Cq^{\frac{n+d+1}{2}},\]
where $d$ is the dimension of the complement of the largest open subset of $\overline{X}_{\overline{\mathbb{F}}_q}$ on which $\overline{X}_{\overline{\mathbb{F}}_q}$ is smooth and $D_{\overline{\mathbb{F}}_q}$ is a smooth divisor ($d=-1$ if the complement is empty).
\end{theorem}
On the transcendental side, we have the following analogue of the above arithmetic results. In the following, $f_{n,q}:\mathbb{P}^n\rightarrow\mathbb{P}^n$ is given by $[x_0:\hdots:x_n]\mapsto [x_0^q:\hdots:x_n^q]$, where $q$ is any positive integer (as opposed to a prime power in the arithmetic cases), and $\Lambda(-)$ denotes the Lefschetz number (see the next section for a definition).
\begin{theorem}\label{C}For each $n\geq 1$, there is a constant $C=C(n)$ such that the following is true. Suppose $X$ is any $n$-dimensional smooth projective complete intersection $X$ over $\mathbb{C}$. Furthermore, suppose it is equipped with a Riemannian metric whose sectional curvature is at least $-\kappa^2$, $\kappa\geq 0$, and whose diameter is at most $D$. Then for each algebraic endomorphism $f:X\rightarrow X$ such that $f^{-1}E$ is algebraically equivalent to $qE$ for some ample divisor $E$ and for some integer $q>0$, 
\[|\Lambda(f)-\Lambda(f_{n,q})|\leq C^{1+\kappa D}q^{n/2}.\]
In particular, there is a constant $\alpha=\alpha(n)$, depending only on the dimension $n$, such that each $f$ as above with $q\geq 2$ has a periodic point of period $\leq \alpha(1+\kappa D)$.
\end{theorem}
\begin{remark}
It is easy to see that $\Lambda(f_{n,q})=1+q+\hdots+q^n$. This is because we know that the cohomology of $\mathbb{P}^n$ is generated by the powers of the hyperplane section.
\end{remark}
\begin{remark}
Note that not all smooth projective complex curves are complete intersections. Indeed, the genus of a degree $(d_1,\hdots,d_n)$ smooth complete intersection in $\mathbb{P}^{n+1}$ (that is a curve) is 
\[g=1+\frac{1}{2}d_1\hdots d_n(-n-2+\sum_id_i).\] 
Consequently, not all genera can appear among the class of smooth projective curves that are complete intersections. For example, it is easy to see from this formula that no such complex curve can have genus $g=2$.
\end{remark}
\begin{remark}
Note that the last part of Theorem~\ref{C} is not true if $q$ is permitted to be $1$. For example, let
\[X_k:=\textup{Proj}\left(\frac{\mathbb{C}[x_0,\hdots,x_n]}{(x_0^{k^n}+\hdots+x_n^{k^n})}\right),\]
and let
\[f_k:X_k\rightarrow X_k\]
be given by
\[[x_0:\hdots:x_k]\mapsto [x_0:\zeta_{k^n}x_1:\hdots:\zeta_{k^n}^nx_n],\]
where $\zeta_{k^n}$ is a primitive $k^n$-th root of unity. Note that each $X_k$ is an $n$-dimensional smooth projective complex variety, and that $f_k$ has the property that
\[f_k^{\circ m}[x_0:\hdots:x_k]=[x_0:\zeta_{k^n}x_1:\hdots:\zeta_{k^n}^{nm}x_n].\]
Since $\zeta_{k^n}^{(j-i)m}\neq 1$ if $k^n\nmid (j-i)m$, $f_k^{\circ m}$ does not have a fixed-point when $m<k^n/n$. Consequently, the minimal period of $f_k$ diverges as $k\rightarrow\infty$.
\end{remark}
\begin{remark}
Note that there is no constant $C$, depending only on $n$, such that for each endomorphism $f$ of any $n$-dimensional smooth projective complex variety, the Lefschetz number $\Lambda(f)$ satisfies
\[|\Lambda(f)-\Lambda(f_{n,q})|\leq Cq^{n/2}.\]
For example, take a genus $g\geq 2$ smooth projective complex curve $\Sigma_g$ (can be endowed with a Riemannian metric of constant negative sectional curvature $-1$) and endomorphisms $\id_{\Sigma_g}:\Sigma_g\rightarrow\Sigma_g$. Then $q=1$ and $\Lambda(\id_{\Sigma_g})=2-2g$. Consequently,
\[\frac{|\Lambda(\id_{\Sigma_g})-1-q|}{q^{1/2}}=2g\rightarrow\infty\]
as $g$ tends to $\infty$. 
\end{remark}
\begin{remark}
Examples of endomorphisms of projective complex varieties with endomorphisms of degree at least $2$ are $\mathbb{P}^n$ and abelian varieties. These are essentially all known smooth examples. By a theorem of Kobayashi and Ochiai from 1975, the number of surjective meromorphic mappings of a compact complex space to a compact complex space of general type is finite \cite{KobayashiOchiai}. In particular, there are no algebraic endomorphisms of degree at least $2$ of smooth projective complex varieties of general type. Of course, this is not true in positive characteristic; the Frobenius endomorphism gives us such endomorphisms. Furthermore, in \cite{Beauville}, Beauville proves that there are no endomorphisms of degree at least $2$ of a smooth hypersurface of degree at least $3$ and dimension at least $2$.
\end{remark}
\begin{remark}
There are no continuous endomorphisms of complete oriented closed connected negatively curved Riemannian manifolds of degree at least $2$ and of finite volume. Indeed, suppose there were such an endomorphism $f:M\rightarrow M$. By Thurston's inequality (see Section 1.2 of Gromov in \cite{Gromov2}), the simplicial volume $|\!|M|\!|$ of $M$ is positive. On the other hand, $|\deg f||\!|M|\!|\leq |\!|M|\!|$ from the definition of simplicial volume (see Section 0.2 of Gromov in \textit{loc.cit}). If $|\deg f|>1$, then these two facts lead to a contradiction. This theorem of Thurston was pointed out to me by Mehdi Yazdi.
\end{remark}
\begin{remark}\label{connectivityremark}
Note that complete intersections of \textit{positive} dimension are (geometrically) connected. See Exercise 18.6.T of \cite{Vakil}, for example.
\end{remark}
\textit{Acknowledgment. }The author is thankful to Mehdi Yazdi for useful discussions and comments. The work of the author during the writing of this paper was supported by SFB1085: Higher Invariants.
\section{Preliminaries}
Before proving our main results, we discuss here the ingredients that will go into proving them. Precisely, after defining some basic relevant notions in Riemannian geometry, we discuss the Lefschetz and Grothendieck fixed-point theorems, the Weil conjectures and Serre's characteristic zero analogue for smooth projective complex varieties, arithmetic and differential geometric theorems bounding from above the sum of Betti numbers, and prove a small lemma about the cohomology groups of complex smooth projective complete intersections.
\subsection{Riemannian Geometry}
For us, the only relevant concepts from Riemannian geometry are sectional curvature and diameter. We briefly define these two concepts. Since this material is standard, we do not give precise citations; see any book on Riemannian geometry, for example Kobayashi and Nomizu's \cite{KN}, for details.\\
\\
Given a manifold $M$ with a vector bundle $E$ above it, we have the notion of connections $D:\cal{T}(M)\otimes\Gamma(E)\rightarrow\Gamma(E)$ on the vector bundle $E$, where $\cal{T}(M)$ denotes vector fields on $M$ and $\Gamma(E)$ denotes global sections of $E$. Its value at $(X,\mu)\in\cal{T}(M)\times\Gamma(E)$ is denoted by $D_X\mu$. A connection is, by definition, one such map that is $C^{\infty}(M)$-linear in $X$ and satisfies a Leibniz rule in $\mu$. Now suppose $M$ is endowed with a metric $g$. Though we do not define what it means for a connection on the Riemannian manifold $(M,g)$ to be symmetric, torsion-free, or compatible with the metric, there is a unique symmetric and torsion-free connection $\nabla$ on (the tangent bundle $TM$ of) $M$ that is compatible with $g$. This is called the Levi-Civita connection. Henceforth, given a vector bundle $V$ on $M$, $\Omega^p(V):=\Omega^p(M)\otimes \Gamma(V)$. Any connection $D$ on a vector bundle $E$ over a manifold $M$ may also be viewed as a map
\[D:\Omega^0(E):=\Gamma(E)\rightarrow \Omega^1(M)\otimes\Gamma(E)=\Omega^1(E)\]
This gives us an operator
\[F:=D^2:\Omega^0(E)\rightarrow\Omega^2(E),\]
called the curvature (endomorphism) of $D$. By duality, this can be viewed as an element of $\Omega^2(\End(E))$, hence the name curvature \textit{endomorphism}. With this identification, $F$ may be viewed as $\mu\mapsto R(.,.)\mu$, where $R(.,.)\in \Omega^2(\End(E))$. It can be shown that $R(.,.)$ is given by 
\[R(X,Y)\mu=D_XD_Y\mu-D_YD_X\mu-D_{[X,Y]}\mu,\]
for every pair of vector fields $X,Y$ and every global section $\mu\in\Gamma(E)$. Let us now restrict to the Levi-Civita connection $\nabla$ with $\Rm$ its curvature endomorphism. This gives us a $4$-tensor which we denote by $(X,Y,Z,W)\mapsto\Rm(X,Y,Z,W):=\left<\Rm(X,Y)Z,W\right>$, where $X,Y,Z,W$ are vector fields on $M$. Suppose $p\in M$ and $X,Y\in T_pM$ is a basis for a $2$-plane $\Pi\subseteq T_pM$. We define the \textit{sectional curvature} as
\[K_g(X,Y):=\frac{\Rm(X,Y,Y,X)}{|X|_g^2|Y|_g^2-\left<X,Y\right>_g^2}.\]
It can be shown that this only depends on the $2$-plane and not on the chosen basis $X,Y$. Therefore, for each $p\in M$, the sectional curvature is a map from the family of $2$-planes in $T_pM$ to $\mathbb{R}$. Intuitively, this quantity is a measure of how curved our manifold is at each point. We could also define the sectional curvature in a more geometric way and then prove the equality above as a theorem. Indeed, there is a geometric definition of the sectional curvature of a surface that associates to each point on a Riemannian surface a number measuring how curved the surface is at that point. For higher dimensional Riemannian manifolds and a given $2$-plane $\Pi\subseteq T_pM$, we can choose a neighbourhood $V$ of $0\in T_pM$ such that the restriction of the exponential map $\exp_p|_V:V\rightarrow M$ is a diffeomorphism onto its image. Then $S_{\Pi}:=\exp_p(V\cap\Pi)$ is a $2$-dimensional Riemannian submanifold of $M$ containing $p$ called the plane section determined by $\Pi$. Then we may apply the geometric version of sectional curvature to the plane section $S_{\Pi}$. This agrees with the above definition of sectional curvature.\\
\\
The definition of diameter is defined in the expected manner. For each piecewise smooth curve $\gamma:I\rightarrow M$ in $M$ defined on some interval $I$, we may define its length as the integral
\[L(\gamma):=\int_I|\gamma'(t)|_gdt.\]
Then the distance between two points $p,q\in M$ is defined as
\[d_g(p,q):=\inf_{\gamma:p\rightarrow q}L(\gamma),\]
where the infimum is taken over all piecewise smooth curves connecting $p$ to $q$. The \textit{diameter} $D_g$ of a Riemannian manifold $(M,g)$ is defined as the supremum of the distance function over all pairs of points in $M$. This concludes our discussion of the two main definitions from Riemannian geometry that will be of interest to us in this paper.
\subsection{Weil conjectures and Serre's theorem}
Before Deligne proved the Weil conjectures, Serre proved a characteristic zero analogue.
\begin{theorem}[Serre \cite{Serre}, 1960]\label{Serrethm}Let $V$ be a smooth projective complex variety, and let $f:V\rightarrow V$ be an endomorphism such that for some hyperplane section $E$, $f^{-1}E$ is algebraically equivalent to $qE$, $q>0$ some integer. Then every eigenvalue of $f^*:H^i(V;\mathbb{C})\rightarrow H^i(V;\mathbb{C})$ is of modulus $q^{i/2}$.
\end{theorem}
If we work over a finite field $\mathbb{F}_q$, then the Frobenius morphism pulls back the hyperplane section $E$ to a divisor algebraically equivalent to $qE$. Consequently, this theorem of Serre is precisely the characteristic zero analogue of the Riemann Hypothesis part of the Weil conjectures, proved by Deligne. In fact, Deligne proved the following more general theorem (in fact, even a relative version).
\begin{theorem}[Deligne \cite{Weil2}, 1980]Let $X_0$ be a scheme of finite type over $\mathbb{F}_q$, and let $\cal{F}_0$ be a $\overline{\mathbb{Q}}_{\ell}$-sheaf on $X_0$ which is $\tau$-mixed of weight $\leq w$. Let $X:=X_0\times_{\mathbb{F}_q}\overline{\mathbb{F}}_q$, and let $\cal{F}$ be the pullback of $\cal{F}_0$ to $X$. Then $H^i_{c,{\textup{\'et}}}(X;\cal{F})$ is $\tau$-mixed of weight $\leq w+i$.
\end{theorem}
We do not define $\tau$-mixed here; see Deligne's \cite{SGA45} and \cite{Weil2}. We note the following consequence. If $X_0$ is a smooth projective $\mathbb{F}_q$-variety, and $\cal{F}_0=\overline{\mathbb{Q}}_{\ell}$, then the above theorem along with Poincar\'e duality give us that $H^i_{\textup{\'et}}(X;\overline{\mathbb{Q}}_{\ell})$ is pure of weight $i$, that is, the geometric Frobenius has eigenvalues of modulus $q^{i/2}$. This is the Riemann Hypothesis over finite fields. If $X_0$ is not projective, then we can only say that the geometric Frobenius acting on $H^i_{c,\textup{\'et}}(X;\overline{\mathbb{Q}}_{\ell})$ has eigenvalues of modulus \textit{at most} $q^{i/2}$.
\subsection{fixed-point theorems}
Suppose $X$ is a topological space with finitely generated singular cohomology groups with coefficients in $\mathbb{C}$, only finitely many of which are nonzero. The \textit{Lefschetz number} $\Lambda(f)$ of a continuous endomorphism $f:X\rightarrow X$ is defined to be
\[\Lambda(f):=\sum_i(-1)^i\tr(f^*:H^i_c(X;\mathbb{C})\rightarrow H^i_c(X;\mathbb{C})).\] 
\begin{theorem}[Lefschetz fixed-point theorem] Let $X$ be a compact triangulable topological space, and let $f:X\rightarrow X$ be a continuous endomorphism. Then if $\Lambda(f)\neq 0$, $f$ has a fixed-point.
\end{theorem}
\begin{remark}
The Lefschetz number $\Lambda(f)$, when the fixed-points of $f$ are non-degenerate, is a signed sum of the fixed-points taking orientations into account. Furthermore, note that the converse is not true. For example, let $f=\id_T$, where $T$ is a torus. Then $\Lambda(\id_T)=\chi(T)=0$, while $\id_T$ clearly has fixed-points.
\end{remark}
The arithmetic analogue of the Lefschetz fixed-point theorem allows us to count the number of points over finite fields of $\mathbb{F}_q$-varieties.
\begin{theorem}[Grothendieck trace formula, Theorem 3.2 of SGA $4\tiny{\frac{1}{2}}$\cite{SGA45}]
Let $X_0$ be an $\mathbb{F}_q$-variety with $\cal{F}_0$ a constructible $\mathbb{Q}_{\ell}$-sheaf on it. Let $X:=X_0\times_{\mathbb{F}_q}\overline{\mathbb{F}}_q$, and let $\cal{F}$ be the pullback of $\cal{F}_0$ to $X$. Then
\[\sum_{x\in X(\mathbb{F}_q)}\tr(F_x;\cal{F}_x)=\sum_i(-1)^i\tr(\Fr^*:H^i_{c,{\textup{\'et}}}(X;\cal{F})\rightarrow H^i_{c,{\textup{\'et}}}(X;\cal{F})),\]
where $\Fr:X\rightarrow X$ is the geometric Frobenius endomorphism, and we are taking compactly supported $\ell$-adic cohomology.
\end{theorem}
As a corollary, we obtain the Grothendieck-Lefschetz fixed-point theorem if we take $\cal{F}_0=\mathbb{Q}_{\ell}$.
\begin{corollary}[Grothendieck-Lefschetz fixed-point theorem] With the above notation, we obtain
\[|X(\mathbb{F}_q)|=\sum_i(-1)^i\tr(\Fr^*:H^i_{c,{\textup{\'et}}}(X;\mathbb{Q}_{\ell})\rightarrow H^i_{c,{\textup{\'et}}}(X;\mathbb{Q}_{\ell})).\]
\end{corollary}
Therefore, one way to obtain estimates on $|X(\mathbb{F}_q)|$ is to understand the compactly supported cohomology of $X$ and apply the Weil conjectures. We discuss the relevant cohomological results in the following subsection.
\subsection{Cohomological results}
It often happens that we understand the cohomology of a smooth projective $\mathbb{F}_q$-variety in a good range. In such situations, we are able to prove square-root cancellation results. One of the tools used in the understanding of the cohomology groups is the following theorem due to N.Katz.
\begin{lemma}[Katz's theorem, appendix to Hooley's \cite{Katzapp}]\label{Katztheorem} If $X$ is an $n$-dimensional projective complete intersection over $\overline{\mathbb{F}}_q$ with singular locus $Z$ of dimension $d$ ($d=-1$ if $Z=\emptyset$), then $H^i_{\textup{\'et}}(X;\mathbb{Q}_{\ell})\simeq\mathbb{Q}_{\ell}(-i/2)$ for $n+d+1<i\leq 2n$, where $\mathbb{Q}_{\ell}(-i/2)=0$ for $i$ odd.
\end{lemma}
This is a consequence of the Lefschetz hyperplane theorem and the semi-perversity of vanishing cycles. Note that when $X$ is smooth, Poincar\'e duality along with Katz's theorem imply that the same is true for $0\leq i<n$. Consequently, for smooth projective varieties, the remaining cohomology group not adressed is the middle cohomology $H^n_{\textup{\'et}}(X;\mathbb{Q}_{\ell})$. For non-projective varieties, we will need the following lemma of Sawin.
\begin{lemma}[Sawin, Lemma 2.4 of \cite{Sawin}]\label{Sawinlemma}Let $\overline{X}$ be a complete intersection in projective space over $\overline{\mathbb{F}}_q$, let $D$ be a hyperplane in $\overline{X}$, and let $X:=\overline{X}\setminus D$. Let $Z$ be the complement of the largest open subset of $\overline{X}$ where $\overline{X}$ is smooth and $D$ is a smooth divisor. Then $H^i_{\textup{\'et},c}(X;\mathbb{Q}_{\ell})=0$ for $\dim Z+\dim X+1<i<2\dim X$. If $\dim Z<\dim X-1$, then $H^{2\dim X}_{\textup{\'et},c}(X;\mathbb{Q}_{\ell})=\mathbb{Q}_{\ell}(-\dim X)$.
\end{lemma}
We will also need a transcendental analogue of the above theorem of Katz, at least one for \textit{smooth} projective complex varieties. In the smooth case, this is a consequence of the Lefschetz hyperplane theorem.
\begin{lemma}\label{complexlemma}
Suppose $X$ is an $n$-dimensional smooth projective complex complete intersection, and suppose $f:X\rightarrow X$ is an endomorphism such that $f^{-1}E$ is algebraically equivalent to $qE$ for some ample divisor $E$ and some integer $q>0$. Then $H^i(X;\mathbb{C})\simeq\mathbb{C}(-i/2)$ for $n<i\leq 2n$, where $\mathbb{C}(-i/2)=0$ for odd $i$ and is the one-dimensional complex vector space on which $f^*$ acts via multiplication by $q^{i/2}$ for even $i$.
\end{lemma}
\begin{proof}
By the Lefschetz hyperplane theorem, $H^i(X;\mathbb{C})$ is generated by powers of the hyperplane class $[\omega]$ when $i>n$. Therefore, $H^i(X;\mathbb{C})=\mathbb{C}\left<[\omega]^{i/2}\right>$ for $i>n$, where $[\omega]^{i/2}=0$ if $i$ is odd. Since $f^*[\omega]=q[\omega]$, $f^*$ acts on $H^i(X;\mathbb{C})$ as multiplication by $q^{i/2}$. Consequently, $H^i(X;\mathbb{C})\simeq\mathbb{C}(-i/2)$ for $i>n$, as required.
\end{proof}
\begin{remark}
Our notation $\mathbb{C}(-i/2)$ is idiosyncratic. Note that though $f$ is omitted from the notation, it depends on $f$. Furthermore, by Poincar\'e duality, only the middle dimension is not understood.
\end{remark}
\subsection{Sum of Betti numbers}
When one wants to estimate exponential sums, in order to understand the error, it is useful to bound from above the sum of the Betti numbers of the variety under consideration. A few theorems, due to Katz, give us bounds on the sum of the Betti numbers of a variety in $\mathbb{P}^N$ in terms of the number of and degrees of the defining equations. Though we will not use such theorems in this paper, one such theorem is as follows. Henceforth, the notation $h^i_?(X;F)$ denotes $\dim_FH^i_?(-;F)$, where $F$ is a field and $H^i_?$ is some cohomology theory.
\begin{theorem}[Katz, corollary to theorem 3 of \cite{Katzbetti}]\label{Katzthm} Over an algebraically closed field $k$, let $X\subseteq\mathbb{P}^N$ be defined by the vanishing of $r\geq 1$ homogeneous equations all of degree at most $d$. Then for any prime $\ell\in k^*$, we have the inequalities
\[\sum_ih^i_{\textup{\'et}}(X;\mathbb{Q}_{\ell})\leq \frac{65}{48}\cdot 2^r(13+4rd)^{N+2}\]
and
\[\sum_ih^i_{\textup{\'et}}(X;\mathbb{Q}_{\ell})\leq 9\cdot 2^r(3+rd)^{N+1}.\]
\end{theorem}
In our case, we will not use such an inequality that depends on the algebraic structure of the variety. We will instead use an inequality that depends on the Riemannian structures which can be put on the complex analytification of the algebraic varieties under consideration. This is accomplished by a theorem of Gromov on sums of Betti numbers of closed connected Riemannian manifolds.
\begin{theorem}[Gromov, 1981, 0.2B of \cite{Gromov}]\label{Gromov}Fix a field $F$. There is a constant $C=C(n,F)$ such that for each $n$-dimensional closed connected Riemannian manifold $M$ with sectional curvature at least $-\kappa^2$, $\kappa\geq 0$, and diameter at most $D$,
\[\sum_ih^i(M;F)\leq C^{1+\kappa D}.\]
\end{theorem}
\begin{remark}\label{CheegerGromoll}
If $M$ is a \textit{complete} Riemannian manifold with nonnegative sectional curvature, then it is homeomorphic to a vector bundle over a closed connected Riemannian manifold of nonnegative sectional curvature. This is a theorem of J.Cheeger and D.Gromoll from 1972 \cite{CheegerGromoll}. Therefore, this reduces the case of complete Riemannian manifolds of non-negative sectional curvature to that of compact manifolds with $\kappa=0$.
\end{remark}
\begin{remark}\label{Ricci}A theorem of G.Wei gives an upper bound on the sum of Betti numbers of a closed connected Riemannian manifold in terms of its Ricci curvature, diameter, and conjugate radius \cite{Wei}. On the other hand, under certain conditions, there are upper bounds on the sum of Betti numbers in terms of the simplicial volume. See \cite{Gromov2} for details. Therefore, it is possible to prove analogues of the theorems in this note relating these quantities to arithmetic over finite fields and algebraic dynamics.
\end{remark}
We end this section with two questions.
\begin{question}
Since the sum of the Betti numbers is independent of the smooth structure, Gromov's bound suggests the following question. For each smooth structure on a closed manifold $M$, consider the quantity
\[\inf_g\left(\left|\min\left\{\inf K_g,0\right\}\right|D_g^2\right),\]
where we vary over Riemannian metrics $g$ on $M$, and where $K_g$ and $D_g$ are its sectional curvature and diameter, respectively. Is this quantity independent of the smooth structure?
\end{question}
\begin{question}
Katz's theorem depends on the algebraic structure of the variety, while Gromov's theorem depends on the Riemannian structure. What is the relation between the two kinds of bounds?
\end{question}
\section{Proofs of the arithmetic theorems}
We are now ready to prove the arithmetic theorems relating Riemannian structures to the number of solutions over finite fields.
\begin{proof}[Proof of Theorem~\ref{A}.]
First note that we know by Lemma~\ref{Katztheorem} that $H^i_{\textup{\'et}}(X_{\overline{\mathbb{F}}_q};\mathbb{Q}_{\ell})=\mathbb{Q}_{\ell}(-i/2)$ for $n<i\leq 2n$. Consequently, the $\ell$-adic cohomology of $X$ agrees with that of $\mathbb{P}^n$ in this range. As a result of this, the Grothendieck-Lefschetz trace formula, and the Weil conjectures, we have that
\begin{eqnarray*}||X(\mathbb{F}_q)|-|\mathbb{P}^n(\mathbb{F}_q)||&=&\left|\sum_{i\leq n}(-1)^i\tr(\Fr^*:H^i_{\textup{\'et}}(X_{\overline{\mathbb{F}}_q};\mathbb{Q}_{\ell})\rightarrow H^i_{\textup{\'et}}(X_{\overline{\mathbb{F}}_q};\mathbb{Q}_{\ell}))-\sum_{\substack{i\leq n\\ i\textup{ even}}}q^{i/2}\right|\\
&\leq& \left(\sum_{i\leq n}h^i_{\textup{\'et}}(X_{\overline{\mathbb{F}}_q};\mathbb{Q}_{\ell})+(n+1)\right)q^{n/2}.\end{eqnarray*}
On the other hand, by proper base-change $X$ satisfies $H^i_{\textup{\'et}}(X_{\overline{\mathbb{F}}_q};\mathbb{Q}_{\ell})\simeq H^i(\tilde{X}^{an};\mathbb{Q}_{\ell})$. As a result,
\[||X(\mathbb{F}_q)|-|\mathbb{P}^n(\mathbb{F}_q)||\leq \left(\sum_{i}h^i(\tilde{X}^{an};\mathbb{C})+(n+1)\right)q^{n/2}.\]
By Gromov's inequality (Theorem~\ref{Gromov} above), we know that
\[\sum_{i}h^i(\tilde{X}^{an};\mathbb{C})\leq B^{1+\kappa D}\]
for some constant $B$ depending only on $n$. In light of Remark~\ref{connectivityremark}, complete intersections of positive dimension are connected, and so the connectivity condition to Gromov's theorem is satisfied. Therefore, we can take $C=C(n)$ large enough so that
\[\sum_{i}h^i(\tilde{X}^{an};\mathbb{C})+n+1\leq C^{1+\kappa D}.\]
Putting this into the above gives us
\[||X(\mathbb{F}_q)|-|\mathbb{P}^n(\mathbb{F}_q)||\leq C^{1+\kappa D}q^{n/2},\]
as required. The rest of the theorem is easily proved from this.
\end{proof}
As a corollary, we obtain Theorem~\ref{main} and Corollary~\ref{corA}. Now, we would like to prove Theorem~\ref{B} that is for $\mathbb{Q}$-varieties that are not necessarily projective, but whose analytifications admit complete metrics of non-negative sectional curvature.
\begin{proof}[Proof of Theorem~\ref{B}.]
By the result of Cheeger and Gromoll \cite{CheegerGromoll} (see Remark~\ref{CheegerGromoll} following the statement of Gromov's theorem~\ref{Gromov}), $X^{an}$ is a vector bundle over a compact Riemannian manifold with nonnegative sectional curvature. Therefore, by Gromov's theorem (Theorem~\ref{Gromov} above), there is a constant $B$ depending only on the dimension $n$ such
\[\sum_ih^i_{c,{\textup{\'et}}}(X_{\overline{\mathbb{F}}_q};\mathbb{Q}_{\ell})=\sum_ih^i_c(X^{an};\mathbb{C})\leq B.\]
Again, by Remark~\ref{connectivityremark}, the connectivity condition in Gromov's theorem is satisfied. The equality of the compactly supported Betti numbers follows from proper base change. Note that since $X$ is smooth, Poincar\'e duality applied to each component implies that the sum of compactly supported Betti numbers is the sum of the usual (non-compactly supported) Betti numbers. As a consequence of Lemma~\ref{Sawinlemma} due to Sawin, $H^i_{c,{\textup{\'et}}}(X_{\overline{\mathbb{F}}_q};\mathbb{Q}_{\ell})=0$ for $n+d+1<i<2n$ and is $\mathbb{Q}_{\ell}(-n)$ for $i=2n$.\\
\\
Consequently, 
\[||X(\mathbb{F}_q)|-q^n|\leq ((n+d+2)+B)q^{\frac{n+d+1}{2}}.\]
Since $d\leq n$, $n+d+2\leq 2n+2$. We can take $C=2n+2+B$. The rest of the theorem is easily proved once we have this inequality. The conclusion follows.
\end{proof}
\section{Proof of the transcendental theorem}
In a way similar to the previous section, we use the results in the preliminary section to prove Theorem~\ref{C}, a transcendental result.
\begin{proof}[Proof of Theorem~\ref{C}.] First note that by Lemma~\ref{complexlemma}, $H^i(X;\mathbb{C})\simeq\mathbb{C}(-i/2)$ for $n<i\leq 2n$. By the Lefschetz fixed-point theorem,
\[\Lambda(f)=\sum_i(-1)^i\tr(f^*:H^i(X;\mathbb{C})\rightarrow H^i(X;\mathbb{C})).\]
The cohomological computation above implies that
\[\Lambda(f)=\sum_{i\leq n}(-1)^i\tr(f^*:H^i(X;\mathbb{C})\rightarrow H^i(X;\mathbb{C}))+\sum_{2n\geq 2i>n}q^{i},\]
and so
\begin{eqnarray*}|\Lambda(f)-\Lambda(f_{n,q})|&=& \left|\sum_{i\leq n}(-1)^i\tr(f^*:H^i(X;\mathbb{C})\rightarrow H^i(X;\mathbb{C}))-\sum_{i\leq n}q^{i/2}\right|\\ &\leq & (n+1)q^{n/2}+\sum_{i\leq n}h^i(X;\mathbb{C})q^{i/2},\end{eqnarray*}
where the last inequality follows from the triangle inequality and Serre's theorem (Theorem~\ref{Serrethm} above) on the modulus of the eigenvalues of $f^*$ acting on different cohomology groups.
On the other hand, Gromov's inequality (Theorem~\ref{Gromov} above) on the sum of Betti numbers implies that there is a constant $C$ depending only on $n$ such that
\[n+1+\sum_{i\leq n}h^i(X;\mathbb{C})\leq C^{1+\kappa D}.\]
Therefore,
\[|\Lambda(f)-\Lambda(f_{n,q})|\leq C^{1+\kappa D}q^{n/2}.\]
In particular, if $\sum_{i\leq n}q^i>C^{1+\kappa D}q^{n/2}$, then $f$ has a fixed-point. This last condition is satisfied if $q>(C^{2/n})^{1+\kappa D}$, for example. Therefore, there is a constant $K$, depending only on $n$, such that for any connected $n$-dimensional smooth projective complete intersection $X$ admitting a Riemannian metric with sectional curvature bounded below by $-\kappa^2$, $\kappa\geq 0$, and diameter bounded above by $D$, and any algebraic endomorphism $f:X\rightarrow X$ such that $f^{-1}E$ is algebraically equivalent to $qE$ for some ample divisor $E$ and integer $q>K^{1+\kappa D}$, $f$ has a fixed-point.\\
\\
If we apply the theorem to the case where $q\geq 2$ and some power $f^{\circ m}$ such that $q^m>K^{1+\kappa D}$, for example, when $m>(1+\kappa D)\log_2 K$, then we see that there is a constant $\alpha$ depending only on $n$ such that any algebraic endomorphism $f$ with $q\geq 2$ has a periodic point of order $\leq \alpha(1+\kappa D)$. 
\end{proof}

\small{\textsc{Department of Mathematics, University of Regensburg, Regensburg, Germany}}\\
\textit{Email address:} \texttt{\small{masoud.zargar@ur.de}}
\end{document}